\documentclass[reqno, 11pt]{amsart}

\usepackage{amsmath, amsthm, amssymb, amsfonts,  physics, enumerate, xcolor, bbm}
\usepackage[alphabetic,backrefs]{amsrefs}
\usepackage{dsfont}

\usepackage[colorlinks=true, urlcolor=blue, citecolor=blue, linkcolor=magenta]{hyperref}
\usepackage{geometry}

\usepackage{todonotes}

\newtheorem{thm}{Theorem}[section]
\newtheorem{prop}[thm]{Proposition}

\theoremstyle{definition}
\newtheorem{defn}[thm]{Definition}

\newtheorem{lemma}[thm]{Lemma}

\newtheorem{rem}[thm]{Remark}
\newtheorem{rems}[thm]{Remarks}

\newtheorem{conv}[thm]{Convention}

\theoremstyle{remark}

\numberwithin{equation}{section}

\newcommand{\G}{\mathcal{G}}

\newcommand{\M}{\mathcal{M}}

\newcommand{\C}{\mathbb{C}}

\begin{document}
\setlength{\parskip}{5pt}

\title{Connectivity for quantum graphs via quantum adjacency operators}

 %\thanks{This research was supported by the Deutsche Forschungsgemeinschaft (DFG, German Research Foundation) under Germany's Excellence Strategy -- EXC 2044 -- 390685587, Mathematics M\"unster -- Dynamics -- Geometry -- Structure, the Deutsche Forschungsgemeinschaft (DFG, German Research Foundation) -- Project-ID 427320536 -- SFB 1442, and ERC Advanced Grant 834267 -- AMAREC, {\color{red} What is the acknowledgement for WIRE?} National Science Center, Poland (NCN) grant no. 2021/43/D/ST1/01446. The project is co-financed by the Polish National Agency for Academic Exchange within the Polish Returns Programme. \includegraphics[scale=0.25 {logoNAWA.png}}

\author[K.\ Courtney]{Kristin Courtney}
     \address{Department of Mathematics and Computer Science, University of Southern Denmark\\ Campusvej 55,
5230 Odense M, Denmark}
     \email{kcourtney@imada.sdu.dk}
     
     \author[P.\ Ganesan]{Priyanga Ganesan}
     \address{Department of Mathematics, University of California, San Diego\\ 9500 Gilman Drive, La Jolla, California 92093, USA}
     \email{pganesan@ucsd.edu}

 \author[M.\ Wasilewski]{Mateusz Wasilewski}
     \address{
Institute of Mathematics of the Polish Academy of Sciences, ul. \'Sniadeckich 8, 00-956, Warszawa, Poland}
     \email{mwasilewski@impan.pl}
     

\begin{abstract}

Connectivity is a fundamental property of quantum graphs, previously studied in the operator system model for matrix quantum graphs and via graph homomorphisms in the quantum adjacency matrix model. In this paper, we develop an algebraic characterization of connectivity for general quantum graphs within the quantum adjacency matrix framework. Our approach extends earlier results to the non-tracial setting and beyond regular quantum graphs. We utilize a quantum Perron–Frobenius theorem that provides a spectral characterization of connectivity, and we further characterize connectivity in terms of the irreducibility of the quantum adjacency matrix and the nullity of the associated graph Laplacian. These results are obtained using the KMS inner product, which unifies and generalizes existing formulations.
\end{abstract}

\maketitle

\section{Introduction}

In recent years, the theory of \textit{quantum graphs} has garnered 
considerable attention due to its deep connections with 
operator algebras, quantum information theory, quantum groups and non-commutative geometry. 
Broadly speaking, quantum graphs are operator theoretic generalizations of classical graphs and they may be described mainly in two different ways.

The first approach models quantum graphs as \textbf{operator systems} or self-adjoint operator subspaces. 
This framework originated in quantum information theory during the study of zero-error communication, where quantum graphs emerged as \textit{confusability graphs} of quantum channels \cite{DSW}. 
These structures, often called \textit{non-commutative graphs} or \textit{matrix quantum graphs}, are realized as operator systems within matrix algebras. 
Independently, related structures appeared in \cite{Weaver-Kuperberg} in the study of \textit{quantum relations} on non-commutative spaces. Building on this, Weaver \cite{Weaver12} developed  a broader operator-algebraic theory of quantum graphs, generalizing the concept of matrix quantum graphs. In this setting, a quantum graph is defined as an operator system satisfying a specific bimodule property over a von Neumann algebra. The operator system model of quantum graphs has enabled the quantum extension of numerous classical graph-theoretic notions, including, graph parameters such as independence numbers and chromatic numbers \cite{DSW, Stahlke, Boreland1, Boreland2}, graph homomorphisms and graph colorings \cite{BraGanHar, TT, BHTT23, BHTT24}, random graph models \cite{ChirWas} and graph products \cite{Meena} in the quantum setting.

Independently, a second model of quantum graphs was introduced in \cite{MRV}, rooted in categorical studies of \textit{quantum sets and quantum functions}. Here, the vertex set of a quantum graph is modeled by a C*-algebra equipped with a distinguished state, which can be interpreted as the function algebra over a non-commutative space. 
The graph itself is defined by a linear operator on the quantum vertex set, known as a \textbf{quantum adjacency matrix}, which generalizes the properties of an adjacency matrix of a classical graph.  
This perspective facilitates a deeper analysis of quantum graphs using tools from C*-algebras, category theory and non-commutative geometry and connects naturally to quantum group theory and quantum symmetries. This model has since been generalized to other settings \cite{Kari1, Matsuda1, Gromada, Daws}, and further unified under the framework of \textit{categorified graphs} using unitary tensor categories \cite{Roberto}. This quantum adjacency matrix model of quantum graphs has led to interesting examples and classifications of small-dimensional quantum graphs \cite{Matsuda1, Gromada}, results linking to quantum (automorphism) groups \cite{MRV2, Kari1, Daws, Was24}, the development of graph algebras and path spaces associated to a quantum space \cite{Kari3, Lara1, Lara2} and spectral-theoretic insights \cite{Ganesan, Matsuda2}.

The equivalence between the two different perspectives of quantum graphs have been proved using different methods \cite{MRV, ChirWas, Daws} in both the tracial and non-tracial setting. Conceptually, the operator system model may be viewed as a quantization of a graph's \textbf{edge set}, while the quantum adjacency matrix model corresponds to a quantization of its \textbf{adjacency matrix}.  While the two descriptions are essentially equivalent, there are several results that we understand well in one framework but not in the other. A prominent example for this is the notion of \textit{connectivity}, which is better understood in the operator system model, but not so in the quantum adjacency matrix model.

In \cite{Swift}, a notion of connectivity was introduced for non-commutative graphs using the operator system model. 
There, a non-commutative graph $S \subseteq M_n$ is defined to be connected if there exists $m \in \mathbb{N}$ such that $S^m = M_n$. With this definition, the authors showed that various quantum graphs, such as the quantum Hamming cubes and quantum expanders, are connected and they established quantum analogues of classical results such as the tree-packing theorem. 
Alternately, \cite{Matsuda2} developed a definition of connectivity in the quantum adjacency matrix model, using
quantum graph homomorphisms. In this setting, a quantum graph is said to be connected if there exists a certain surjective quantum graph homomorphism into the completely disconnected classical graph on two vertices. Using this notion, 
the author provided algebraic characterizations of connectedness and bipartiteness of regular tracial quantum graphs using the spectrum of the quantum adjacency matrix. Although notions of connectivity have been defined in both models of a quantum graph, the definition in the operator system model is more tractable than the other. However, both notions of connectivity remain limited in scope, and the relation between the two notions has been unclear.
In particular, the results in \cite{Swift} focus only on matrix quantum graphs, while the results in \cite{Matsuda2} are confined to regular quantum graphs in the tracial setting. 

In this paper, we bridge this gap by introducing an \textbf{algebraic definition of connectivity} in the quantum adjacency matrix model. We prove that this formulation is equivalent to the various existing notions of connectivity and provides a unified approach to connectivity for general quantum graphs. Our approach extends the results in \cite{Matsuda2} to non-tracial settings and beyond the case of regular quantum graphs. In particular, we present a quantum Perron-Frobenius theorem for quantum graphs, thus providing a spectral characterization of connectivity for general quantum graphs. Our results are achieved via the use of the KMS inner product on the quantum set, instead of the more commonly used GNS inner product. This choice enables us to formulate and prove results in the non-tracial setting and also yields simplified proofs for several known results. We also study bipartite quantum graphs, generalizing results from \cite{Matsuda2} to nonregular and nontracial cases and answer a question from \cite{Matsuda2} about the operator norm of a $d$-regular quantum graph.

The paper is organized as follows. In section 2, we introduce some preliminaries about quantum graphs, KMS structure and introduce some lemmas about irreducible maps for future use. In section 3, we present an algebraic definition of connectivity and prove that it is equivalent to various other notions of connectivity in the tracial setting. We also characterize connectivity in terms of irreducibility of the quantum adjacency matrix and show that most (random) quantum graphs are connected. In section 4, we focus on the non-tracial setting and prove the results using KMS implementations. Here, we also show a characterization of connectivity in terms of the nullity of the graph Laplacian. We also discuss bipartite quantum graphs and operator norm of a quantum adjacency matrix in this section.

\section{Preliminaries}\label{sec:prelim}

Let $\M$ be a finite dimensional C*-algebra equipped with a faithful positive functional $\psi$. If $\M \simeq \bigoplus_{a=1}^{d} M_{n_{a}}$ then $\psi$ is given by $\psi = \bigoplus_{a=1}^{d} \mathrm{Tr}(\rho_{a} \cdot)$ (where $\rho_a\in M_{n_a}$ is a positive definite matrix such that $\psi|_{M_{n_a}}=\mathrm{Tr}(\rho_a\cdot)$). In this context we can define the \emph{modular group} of $\psi$ to be $\sigma_z(\bigoplus_{a=1}^{d}x_a) := \bigoplus_{a=1}^{d} \rho_{a}^{iz} x_a \rho_{a}^{-iz}$ for any $z\in \C$.

Using $\psi$ we can define two kinds of inner products on $\M$, the usual GNS inner product $\langle x, y\rangle:= \psi(x^{\ast} y)$ and the KMS inner product $\langle x, y\rangle_{\mathrm{KMS}}:= \psi(x^{\ast} \sigma_{-\frac{i}{2}}(y))$. There are more explicit formulas available in our case, namely
\begin{align*}
\langle \bigoplus_{a=1}^{d} x_a, \bigoplus_{a=1}^{d}y_a\rangle &= \sum_{a=1}^{d} \Tr(\rho_{a} x_{a}^{\ast} y_{a}) \\
\langle \bigoplus_{a=1}^{d} x_a, \bigoplus_{a=1}^{d}y_a\rangle_{\mathrm{KMS}} &= \sum_{a=1}^{d} \Tr(\rho_{a}^{\frac{1}{2}} x_{a}^{\ast} \rho_{a}^{\frac{1}{2}} y_{a}) = \sum_{a=1}^{d} \Tr((\rho_{a}^{\frac{1}{4}} x_{a} \rho_{a}^{\frac{1}{4}})^{\ast} \rho_{a}^{\frac{1}{4}} y_{a} \rho_{a}^{\frac{1}{4}}). 
\end{align*}

We will discuss the usefulness of the KMS inner product in Section \ref{Sec:nontracial}.

From now on we assume that $\M$ is equipped with a $1$-form $\psi$, i.e. a faithful positive functional such that $mm^{\ast}=id$, where $m: \M \otimes \M \to \M$ is the multiplication map and $m^{\ast}: \M \to \M \otimes \M$ is its adjoint with respect to the GNS inner product coming from $\psi$. Note that $\sigma_z$ is a multiplicative map for $z\in C$, so $m \circ(\sigma_z \otimes \sigma_z) = \sigma_z \circ m$, that is $m$ commutes with the respective modular groups of $\M\otimes \M$ and $\M$, hence the its adjoint is the same with respect to either the GNS or the KMS inner product. Such a pair $(\M, \psi)$ is sometimes called a \emph{quantum space}. If $\M \simeq \bigoplus_{a=1}^{d} M_{n_{a}}$ and $\psi$ is given by $\psi = \bigoplus_{a=1}^{d} \mathrm{Tr}(\rho_{a} \cdot)$, then the condition $mm^{\ast} = id$ is equivalent to $\mathrm{Tr}(\rho_{a}^{-1}) = 1$ for all $a\in \{1,\dots, d\}$. One can verify by a direct computation (see, e.g., \cite[Proposition 2.3]{Was24}) that $m^{\ast}(e_{ij}^a)=\sum_{k=1}^{n_a}e^a_{ik}\rho_a^{-1}\otimes e^a_{kj}$ where $e_{ij}^a\in M_{n_a}$ are the matrix units. 

Using the maps $m$ and $m^{\ast}$ we can define the \emph{quantum Schur product} of linear maps on $\M$, namely if $A,B: \M \to \M$ then we define $A\bullet B:= m(A\otimes B)m^{\ast}$. If $\M = \C^n$ then this is the usual Schur product, which we will denote by $\odot$, i.e. $(S\odot T)_{ij}:= S_{ij} T_{ij}$ for two matrices $S,T \in M_n$.

\begin{defn}[\cite{MRV}]
A quantum graph on $(\M, \psi)$ is defined by a linear operator $A: \mathcal{M} \to \mathcal{M}$, which is a quantum Schur idempotent, that is it satisfies
\[ 
m(A \otimes A)m^* = A, 
\]
 We call $A$ a \textit{quantum adjacency matrix}. The triple $\G:= (\M, \psi, A)$ will be called a \textit{quantum graph}.

 The quantum graph $\G$ is said to be 
\begin{itemize}
\item \emph{Tracial} if $\psi$ is tracial;
\item \emph{Real} if $A$ is a *-preserving map on $\M$.
\item \emph{Undirected} if $A$ is self-adjoint with respect to the KMS inner product (in the tracial case it coincides with the GNS inner product). We will often write that $A$ is KMS-symmetric.
\item \emph{Reflexive} if $A \bullet id = id$ and \emph{irreflexive} if $A \bullet id = 0$.
\end{itemize}
We will assume that our quantum graphs are real.
\end{defn}

There is another way to encode a linear map on $\M$, namely by using its Choi matrix, which is an element of $\M \otimes \M^{\mathrm{op}}$.
\begin{defn}\label{def:Choi}
Let $\G = (\M, \psi, A)$ be a quantum graph. We call $\mathrm{Choi}(A):= (A \otimes \sigma_{-\frac{i}{2}})m^{\ast}(\mathds{1}) \in \M \otimes \M^{\mathrm{op}}$ the \emph{Choi matrix} of $A$. We can write it more explicitly as (see \cite[Lemma 3.3]{Was24})
\[
\mathrm{Choi}(A) = \sum_{a=1}^{d}\sum_{i,j=1}^{n_{a}} A(\rho_{a}^{-\frac{1}{4}}e_{ij}^{a} \rho_{a}^{-\frac{1}{4}}) \otimes (\rho_{a}^{-\frac{1}{4}}e_{ji}^{a} \rho_{a}^{-\frac{1}{4}})^{\mathrm{op}}.
\]
\end{defn}
we will summarize here the most important properties of the Choi matrix.
\begin{prop}[{\cite[Proposition 3.7]{Was24}}]\label{prop:Choiproperties}
Let $A,B: \M \to \M$ be linear maps, where $\M$ is equipped with a $1$-form $\psi$.

\begin{itemize}
\item $\mathrm{Choi}(A\bullet B) = \mathrm{Choi}(A) \mathrm{Choi}(B)$.
\item $\mathrm{Choi}(A)$ is self-adjoint if and only if $A$ is a $\ast$-preserving map.
\item $\mathrm{Choi}(A)$ is positive if and only if $A$ is a completely positive map.
\item $\mathrm{Choi}(A) = \Sigma (\mathrm{Choi}(A))$, where $\Sigma(a\otimes b^{\mathrm{op}}):= b\otimes a^{\mathrm{op}}$ is the tensor flip, if and only if $A = A^{T}$, where the transpose is defined as $A^{T}(x) := (A^{\ast}_{\mathrm{KMS}}(x^{\ast}))^{\ast}$
\end{itemize}
\end{prop}
\begin{rem} Reality of $A$ is equivalent to complete positivity, because the Choi matrix $P:= \mathrm{Choi}(A)$ is an idempotent, so if it is self-adjoint, then $P = P^{\ast}P\geqslant 0$, hence $A$ has to be completely positive.
\end{rem}

There is another description of quantum graphs in terms of quantum relations introduced in \cite{Weaver12}. To wit, whenever $\M$ is faithfully represented inside $B(H)$, then quantum relations are described by weak${^\ast}$ closed $\M'$-bimodules inside $B(H)$; in a precise sense, it does not depend on the choice of the representation of $\M$. We will use an explicit correspondence from \cite[Proposition 3.30]{Was24}.
\begin{prop}\label{prop:qadjacencyfrombimodule}
Let $\M \simeq \bigoplus_{a=1}^{d} M_{n_{a}}$ be equipped with a $1$-form $\psi = \bigoplus_{a=1}^{d} \mathrm{Tr}(\rho_{a} \cdot)$ and represented inside $B(H)$, where $H := \bigoplus_{a=1}^{d} \mathbb{C}^{n_a}$. Then any $\M'$-bimodule $S$ is a collection of subspaces $S_{ab} \subset B(\mathbb{C}^{n_{a}}, \mathbb{C}^{n_{b}})$. The corresponding quantum adjacency matrix $A: \M \to \M$ can be written as $\bigoplus_{a,b=1}^{d} A_{ab}$, where $A_{ab}: M_{n_{a}} \to M_{n_{b}}$ is given by
\[
A_{ab} (x) = \sum_{i} \rho_{b}^{-\frac{1}{4}} S^{i}_{ab} \rho_{a}^{\frac{1}{4}} x \rho_{a}^{\frac{1}{4}} (S^{i}_{ab})^{\ast} \rho_{b}^{-\frac{1}{4}},
\]
where $(S^{i}_{ab})_{i}$ is an orthonormal basis of $S_{ab}$ with respect to the KMS inner product induced by $\psi^{-1}$ (see \cite{Was24} for more details). Note that in the tracial case the expression simplifies significantly giving a Kraus decomposition.
\end{prop}
 
\begin{conv}
A projection $p \in \M$ refers to an element satisfying $p = p^2  = p^*$; we say it is non-trivial if $p\neq 0$ and $p\neq \mathds{1}$. The left multiplication by $p$ in $B(L^2(\M, \psi))$ is denoted by $L_p$ and is also a projection. The right multiplication will be denoted by $R_{p}$.
\end{conv}

We record here the following lemma for further use.
\begin{lemma}\label{Lem:irred}
Let $\Phi: \M \to \M$ be a completely positive map. Fix a non-trivial projection $p \in \M$. The following are equivalent:
\begin{enumerate}
\item\label{irr1} $\Phi(p) \leqslant Cp$ for some constant $C>0$, i.e. $\Phi$ is \emph{reducible};
\item\label{irr2} $\Phi(p\M p) \subset p\M p$;
\item\label{irr3} $\Phi(xp) = \Phi(xp)p$;
\item\label{irr4} $\Phi(px) = p \Phi(px)$;
\item\label{irr5} $\Phi(p) (1-p) = 0$.
\end{enumerate}
If $\M$ is equipped with a faithful, tracial positive functional and $\Phi$ is self-adjoint with respect to this functional then any of these conditions is equivalent to $\Phi(xp) = \Phi(x)p$ or, equivalently, $\Phi(px) = p \Phi(x)$.
\end{lemma}
\begin{proof}
We start with \eqref{irr1}. If $\Phi(p) \leqslant Cp$ then $(1-p) \Phi(p) (1-p)  =0$, from which we infer that $\sqrt{\Phi(p)}(1-p) = 0$, hence $\Phi(p)(1-p)=0$, i.e. \eqref{irr5} holds. It follows from the Kadison-Schwarz inequality that $(\Phi(xp))^{\ast} \Phi(xp) \leqslant \|\Phi(\mathds{1})\| \Phi(px^{\ast}xp) \leqslant  \|\Phi(\mathds{1})\|\|x\|^2 \Phi(p)$. From $\Phi(p)(1-p)=0$ it follows that $(1-p) (\Phi(xp))^{\ast} \Phi(xp) (1-p)=0$, hence $\Phi(xp)(1-p) = 0$, i.e. $\Phi(xp) = \Phi(xp)p$, so \eqref{irr5} implies \eqref{irr4}. This equality is equivalent to $\Phi(px) = p\Phi(px)$ as $\Phi$ is a $\ast$-preserving map, hence \eqref{irr3} is equivalent to \eqref{irr4}. It follows that $\Phi(pxp) = p\Phi(pxp) p$, hence $\Phi(p\M p) \subset p\M p$, thus \eqref{irr2} is implied by the combination of \eqref{irr4} and \eqref{irr3}. From this it clearly follows that $\Phi(p) \leqslant Cp$, as any positive element in $p \M p$ is of the form $pxp$ for some $x\in \M_{+}$, hence it is bounded above by $\|x\| p$, hence we deduced \eqref{irr1}, which finishes the proof that all these conditions are equivalent. 

Now we assume that $\Phi$ is self-adjoint. We use the condition $\Phi(px) = p\Phi(px)$, which can be written as $\Phi L_p = L_p \Phi L_p$, where $L_p$ is the left multiplication by $p$. As $(L_p)^{\ast} = L_p$, the right-hand side is self-adjoint, so it follows that $\Phi L_p = (\Phi L_p)^{\ast} = L_p \Phi $, i.e. $\Phi(px) = p\Phi(x)$. As $\Phi$ is $\ast$-preserving, this is equivalent to $\Phi(xp) = \Phi(x)p$, i.e. $\Phi R_p = R_p \Phi$.
\end{proof}

We will need the non-commutative Perron-Frobenius theorem.
\begin{thm}[{\cite[Theorem 2.3, Theorem 2.5]{PF-cp}}]
Let $\Phi$ be a positive map on a finite dimensional $\mathrm{C}^{\ast}$-algebra $\M$. Let $r$ be the spectral radius of $\Phi$. Then there exists a positive eigenvector $x \in \M$ of $\Phi$ with eigenvalue $r$. If $\Phi$ is irreducible then $r$ is a simple eigenvalue and the eigenvector $x$ is strictly positive, i.e. it is invertible; it is then called the Perron-Frobenius eigenvector of $\Phi$. Moreover $x$ is the unique positive eigenvector of $\Phi$.
\end{thm}
\begin{rem}
The converse may not be true, i.e. $r$ could be a simple eigenvalue yet $\Phi$ may be reducible. Indeed, $\Phi$ can split as a direct sum of two irreducible maps $\Phi_1$ and $\Phi_2$ such that the spectral radius of $\Phi_1$ is strictly bigger than that of $\Phi_2$. Then the Perron-Frobenius eigenvector of $\Phi_1$ will be the unique eigenvector of the highest eigenvalue.

Even if we know that $r$ is a simple eigenvalue and the corresponding eigenvector is strictly positive then the map $\Phi$ can fail to be irreducible. Consider the matrix $S:=\left(\begin{array}{cc} 1 & 1 \\ 0 & 2 \end{array}\right)$ acting on $\mathbb{C}^2$. It has eigenvalues $1$ and $2$ and the eigenvector corresponding to the eigenvalue $2$ is a multiple of $(1,1)$, so it is strictly positive. The matrix $S$ is not irreducible, because the projection $e_1 \in \mathbb{C}^2$ is an eigenvector of $S$.

In case that $\Phi$ is self-adjoint with respect to some trace, Lemma \ref{Lem:irred} shows that if $r$ is a simple eigenvalue and the corresponding eigenvector is strictly positive, then $\Phi$ has to be irreducible. Indeed, if $\Phi$ was not irreducible then for some non-trivial projection $p\in \M$, we would have $\Phi(xp) = \Phi(x)p$. Let $y$ be the the strictly positive eigenvector, i.e. $\Phi(y) = ry$. It follows that $\Phi(yp) = ryp$, so $yp$ is another eigenvector with eigenvalue $r$; $yp \neq 0$, because $y$ is strictly positive.
\end{rem}
\newpage

\section{Connectivity}
In this section we assume that $\psi$ is tracial. 

The following definition is motivated by the fact that if $G$ is a disconnected classical graph, then the adjacency matrix of $G$ can be expressed as a block diagonal matrix with each block corresponding to a connected component of the graph.

\begin{defn}[Connected quantum graphs] \label{definition: connectivity 1}
Let  $\G = (\M, \psi, A)$ be a real tracial quantum graph.
\begin{enumerate}
\item 
An \textit{undirected} quantum graph $\G$ is said to be \textit{connected} if the quantum adjacency matrix $A$ does not commute with any non-trivial projection $p \in \M$, i.e.
\begin{equation}
L_p A = A L_p \implies p = 0 \text{ or } \mathds{1}.
\end{equation}
Otherwise, we say $\G$ is disconnected.
\item A \textit{directed} quantum graph $\G$ is said to be \textit{strongly connected} if there does not exist any non-trivial projection $p \in \M$ such that 
$L_p A L_p = A L_p$. 
%Otherwise, we say $\G$ is disconnected.
\item $\G$ is said to be \textit{totally disconnected} if $A \in \M' \subseteq B(L^2(\M, \psi))$. 
\item If there exists a non-trivial projection $p \in \M$ such that $L_p A  = A L_p$ and $p$ is a minimal such projection, then we say that  $p\M p$ is a connected component of $\G$ (in the language of \cite[Section 6]{Weaver} it is the restriction of the quantum graph to $p\M p$).  \textcolor{blue}{Equivalently, $L_p$ is the orthogonal projection onto a connected component of $\G$.} 
\end{enumerate}
\end{defn}

\begin{rems}
\begin{enumerate}
    \item The term \textit{strongly connected} refers to connectivity of \textit{directed} quantum graphs, while simply the term \textit{connected} refers to the case of \textit{undirected} quantum graphs. Both are equivalent to irreducibility of the quantum adjacency matrix $A$ by Lemma \ref{Lem:irred}. 
    \item  In general, if $A$ and $P$ are two bounded operators on a Hilbert space $\mathcal{H}$ with $P$ a projection, the relation $AP=PAP$ says exactly that $AP\mathcal{H}\subset P\mathcal{H}$, i.e., the range of $P$ is an invariant subspace for $A$. When $A$ is self-adjoint, this says that a projection $P$ commutes with $A$ if and only if its range is an invariant subspace for $A$. 
    \item In finite dimensions, any eigenspace of $A$ is an invariant subspace of $A$ and hence the orthogonal projection $P$ onto an eigenspace of $A$ will satisfy $(I-P)AP = 0$. Therefore, in order to define connectivity, we must restrict the choice of $P$, and the projections $p$ from the quantum set $\M$ turn out to be a good candidate for this.
\end{enumerate}

\end{rems}

% In the classical case, it is known that if $A$ is self-adjoint, then 
% a projection $P$ commutes with $A$ $\iff$ Range$(P)$ is an invariant subspace for $A$, as noted below:
% \begin{prop}
% $(I-P)AP  = 0 \iff$ Range($P$) is an invariant subspace for $A$.
% \end{prop}
% \begin{proof}
% First note that $AP = PAP$ implies $Range(AP) \subseteq Range(P)$. 
% \begin{itemize}
% \item[$(\Leftarrow)$] Let $R$ be an invariant subspace for $A$ and $P$ be projection into $R$.
% Then, $Px = x$ for all $ x \in R$. Thus
% \[P(APx) = A (Px) = Ax \in R\]
% Hence, $PAP = AP$ and $(I-P)AP = 0$.
% \item[$(\Rightarrow)$] Let $R$ be the range of $P$ and let $\tilde{P}$ be the orthogonal projection onto $R$.
% Then, $R = \{ Pz: z \in L^2(\M)\}$.
% \[ A(Pz) = PAPz = P(APz) \in R\]
% Thus, $R$ is an invariant subspace for $A$.
% \end{itemize}
% \end{proof}

% This motivates the following definition: 
% \begin{defn}
% If there exists a non-trivial projection $p \in \M$ such that $L_p A  = A L_p$, then we say that   \textcolor{red}{Range $(L_p)$ or $p\M p$} is a connected component of $\G$. Equivalently, $L_p$ is the orthogonal projection onto a connected component of $\G$.
% \end{defn}

% \begin{rem}
% In finite dimensions, any eigenspace of $A$ is an invariant subspace of $A$ and hence the orthogonal projection $P$ onto an eigenspace of $A$ will satisfy $(I-P)AP = 0$, as noted above. Therefore, in order to define connectivity, we must restrict the choice of $P$ and the projections $p$ from the quantum set $\M$ turn out to be a good candidate for this.
% \end{rem} 

We now show that the above definition captures the classical notion of connectedness.

\begin{prop}
Let $G$ be a simple finite classical graph on $n$ vertices and adjacency matrix $A$. Then, $G$ is connected if and only if the corresponding non-commutative graph $\G = (D_n, \Tr, A)$ is connected in the sense of definition \ref{definition: connectivity 1}.
\end{prop}
\begin{proof}
$\G$ is disconnected if and only if there is a non-trivial subset $X \subset [n]$ such that $a_{ij} = 0$ whenever $i\in X$ and $j \in [n]\setminus X$. This means exactly that $A = L_p A L_p + (\mathds{1}-L_p)A(\mathds{1}-L_p)$, where $p \in D_n$ is the orthogonal projection corresponding to $X$. By self-adjointness of $A$ it is equivalent to the condition $L_p A = A L_p$. 
\end{proof}

We next show that Definition \ref{definition: connectivity 1} agrees with and unifies the existing notions of connectivity for quantum graphs in the literature.

\begin{thm}\label{Thm:mainequivalence}
Let $\G = (\M, \psi, A, S)$ be an undirected, real, tracial quantum graph. Let $T_2$ denote the classical totally disconnected graph on two vertices. The following are equivalent:
\begin{enumerate}
\item\label{main1} $A$ is a reducible completely positive map.
\item\label{main2} There exists a non-trivial projection $p \in \M$ such that $L_pA = A L_p$.
\item\label{main3} There exists a non-trivial projection $p \in \M$ such that $(I-L_p) S L_p = 0$.
\item\label{main4} $C^*(S) \neq B(L^2(\M, \psi))$. In particular, for matrix quantum graphs $(M_n, n\mathrm{Tr}, A)$, we have $S^k \neq M_n$ for all $k \in \mathbb{N}$, where $S\subset M_n$ is the self-adjoint subspace corresponding to $\G$.
\item\label{main5} The sequence of projections $(p_1, p_2, p_3 \ldots ) \in \M \otimes \M^{op}$, where $p_k$ corresponds to the quantum relation $S^{k}$ (or, equivalently, $p_k$ is the support projection of the Choi matrix of $A^{k}$) satisfies $\sup_{k} p_k \neq \mathds{1}$. In the reflexive case, this is an increasing sequence $p_1 \le p_2 \le p_3 \ldots$.
\item\label{main6} There exists an injective unital *-homomorphism $f: \C^2 \to \M$ satisfying 
\[ \begin{bmatrix} 1 & 0 \\ 0 & 1 \end{bmatrix} \odot (f^*Af) = (f^*Af), \]
where $\odot$ denotes the Schur product of two matrices. \item\label{main7} There exists a $(loc)$-graph homomorphism from $\G$ to $T_2$ in the sense of \cite{BraGanHar}.
\item\label{main8} There exists a surjective graph homomorphism from $\G$ to $T_2$ in the sense of \cite{MRV}.
\item\label{main9} $0$ is a simple eigenvalue of the quantum graph Laplacian $\Delta$ (see \cite[Remark 2.7]{Matsuda2}).
\end{enumerate}

\end{thm}
We need some preparation for the proof. First, we discuss the directed case, where we prove that irreducibility of the quantum adjacency matrix $A$ is equivalent to the fact that the algebra generated by the corresponding quantum relation $S$ is equal to $B(H)$ on which $\M$ is represented, regardless of the representation. We begin with a lemma.
\begin{lemma}\label{lem:Burnside}
Let $\G:=(\M, \psi, A)$ be a (possibly directed) quantum graph. 
%Let $\pi_1: \M \to B(H_1)$ and $\pi_2: \M \to B(H_2)$ be two faithful representations and let $S_1 \subset B(H_1)$ and $S_2\subset B(H_2)$ be the corresponding $\M'$-bimodules. The following conditions are equivalent for $S_1$ and $S_2$ (we refer to either of them by $S$):
For each of the following conditions, if the condition holds for some faithful representation $\pi:\M\to B(H)$ and corresponding $\M'$-bimodule $S\subset B(H)$, then it holds for any faithful representation and corresponding $\M'$-bimodule.
\begin{enumerate}
\item\label{Burnside} The algebra generated by $S$ is not equal to $B(H)$;
\item\label{Invariantsubspace} There exists a non-trivial projection $P\in B(H)$ such that $(\mathds{1} - P)SP = 0$;
\item\label{Disconnected} There exists a non-trivial projection  $p \in \pi(\M)$ such that $(\mathds{1} - p)Sp=0$.
\end{enumerate}
Moreover, conditions \eqref{Burnside} and \eqref{Invariantsubspace} are equivalent.
\end{lemma}
\begin{proof}
It follows from Weaver's proof that the definition of a quantum relation does not depend on the choice of a representation \cite[Theorem 2.7]{Weaver12}. %(It is the same as in the arXiv version). 
Note that it is clear that these conditions are equivalent for a representation $\pi$ and its inflation $\pi\otimes \mathds{1}_{\ell^2}: \M \to B(H\otimes \ell^2)$, because then the inflated $S$ will be $S \otimes B(\ell^2) $. However, for two faithful representations $\pi_1$ and $\pi_2$ the inflations $\pi_1 \otimes \mathds{1}_{\ell^2}$ and $\pi_2\otimes \mathds{1}_{\ell^2}$ are unitarily equivalent, so one can transfer all of these properties.

Finally, the equivalence of \eqref{Burnside} and \eqref{Invariantsubspace} follows from Burnside's theorem, saying that any proper subalgebra of a matrix algebra admits a non-trivial invariant subspace.
\end{proof}
We will now choose a specific representation, for which it is easy to prove that the conditions \eqref{Invariantsubspace} and \eqref{Disconnected} are also equivalent.
\begin{lemma}\label{lem:projectioninalgebra}
Let $\M := \bigoplus_{a=1}^{d} M_{n_a}$ be a finite dimensional $\mathrm{C}^{\ast}$-algebra, represented on the Hilbert space $H:= \bigoplus_{a=1}^{d} \mathbb{C}^{n_a}$. Let also $S \subset B(H)$ be an $\M'$-bimodule. If there exists a non-trivial projection $P \in B(H)$ such that $(\mathds{1} - P)SP=0$ then there also exists a non-trivial projection $p\in \M$ such that $(\mathds{1} - p)Sp=0$.
\end{lemma}
\begin{proof}
Let $B$ be the (possibly non-unital) algebra generated by $S$, and note that for any projection $P\in B(H)$, $(\mathds{1} - P)SP=0$ if and only if $(\mathds{1} - P)BP=0$. Suppose that $\mathds{1} \in B$, then $\M' \subset B$, since $B$ is an $\M'$-bimodule. If the projection $P \in B(H)$ provides an invariant subspace for $B$, then it also has to be invariant for $\M' \subset B$, hence $P \in \M$. 

So there might be a problem coming from non-unitality. We will show that if $B\subset B(H)$ is a proper subalgebra then $\widetilde{B}:= B + \M'$ is also a proper subalgebra of $B(H)$, which contains $\M'$. It will follow from the unital case that there exists a non-trivial projection $p \in \M$ such that $(\mathds{1} -p)\widetilde{B}p=0$, in particular $(\mathds{1} - p)Sp=0$. Suppose that $\widetilde{B} = B(H)$. In this particular representation we have $\M' = \text{span}\{\mathds{1}_{n_{a}}: a \in \{1,\dots, d\}\}$, i.e. it is spanned by the units of the blocks of $\M$. We define $S_{ab}:= \mathds{1}_{n_b} S \mathds{1}_{n_{a}}$. Because $\mathds{1} = \sum_{a=1}^{d} \mathds{1}_{n_{a}}$, we conclude that $S = \bigoplus_{a,b=1}^{d} S_{ab}$. Therefore, any $\M'$-bimodule $S$ splits as a direct sum of blocks $S_{ab} \subset B(\mathbb{C}^{n_{a}}, \mathbb{C}^{n_{b}})$. Let us take a look at any element $x$ from the $ab$-block in $B(H)$ for $a\neq b$. As $\widetilde{B} = B(H)$, we have $x \in B + \M'$, i.e. $x=b+m'$. Since $x$ belongs to the $ab$-block, we have $x = \mathds{1}_{n_{b}}x \mathds{1}_{n_{a}}$, so $x = \mathds{1}_{n_{b}}b\mathds{1}_{n_{a}} + \mathds{1}_{n_{b}}m'\mathds{1}_{n_{a}}$. Note that $\mathds{1}_{n_{b}}b\mathds{1}_{n_{a}} \in B$ by the $\M'$-bimodule property and $\mathds{1}_{n_{b}}m'\mathds{1}_{n_{a}}=0$, because $\M'$ is contained in the diagonal blocks. It follows from this argument that for any $a\neq b$ the whole $ab$-block of $B(H)$ is contained in $B$. But $B$ is an algebra and you can easily express any element in the $aa$-block as a product of elements in the blocks $ab$ and $ba$, so we conclude that $B = B(H)$.

To sum up, if $(\mathds{1} - P)S P=0$ for some non-trivial $P \in B(H)$, then the algebra generated by $S$ and $\M'$ is a proper subalgebra of $B(H)$, so there exists a non-trivial projection $p \in \M$ such that $(\mathds{1} - p)S p = 0$.
\end{proof}

\begin{prop}\label{Prop:irredbimod}
Let $\mathcal{G}:=(\M, \psi, A)$ be a quantum graph. Then $A$ is irreducible if and only if the algebra generated by the corresponding quantum relation $S\subset B(H)$ is equal to $B(H)$. It means that a quantum graph is strongly connected if and only if the algebra generated by the quantum relation $S$ is equal to $B(H)$.
\end{prop}
\begin{proof}
We already know from the previous lemmas that we can choose a specific representation $\M \subset B(H)$ and for $\M \simeq \bigoplus_{a=1}^{d} M_{n_a}$ we choose $H:= \bigoplus_{a=1}^{d} \mathbb{C}^{n_a}$. We already established that the algebra generated by $S$ is equal to $B(H)$ if and only if there is no non-trivial projection $p\in \M$ such that $(\mathds{1} - p)Sp = 0$.

Suppose that there exists a $p\in \mathcal{M}$ such that $(\mathds{1} - p)Sp=0$. As $\mathcal{M} \simeq \bigoplus_{a=1}^{d} M_{n_{a}}$, we have $p = \bigoplus p_{a}$. We conclude that $(\mathds{1}_{n_{b}} - p_b)S_{ab} p_{a} = 0$. The associated quantum adjacency matrix $A: \mathcal{M} \to \mathcal{M}$ splits into blocks $A = \bigoplus_{a,b=1}^{d} A_{ab}$, where $A_{ab}: M_{n_{a}} \to M_{n_{b}}$. Recall the formula for $A_{ab}$ from Proposition \ref{prop:qadjacencyfrombimodule} (used in the tracial case here):
\[
A_{ab} (x) = \sum_{i} S_{ab}^{i} x (S_{ab}^{i})^{\ast},
\]
where $(S_{ab}^{i})_{i}$ is an orthonormal basis of $S_{ab}$. It follows from our assumptions that each $S_{ab}^{i}$ satisfies $S_{ab}^{i} p_{a} = p_{b} S_{ab}^{i} p_{a}$, so we get that $A_{ab}(p_{a} x p_{a}) = p_{b} A_{ab}(p_{a} x p_{a}) p_{b}$, which means that $A(pxp) = pA(pxp)p$, hence $A$ is not irreducible.

Assume now that $A$ is not irreducible, that is $A(pxp) = pA(pxp)p$ for some non-trivial projection $p\in \mathcal{M}$. Splitting it into blocks, we obtain two Kraus decompositions
\[
\sum_{i} S_{ab}^{i} p_{a}x p_{a} (S_{ab}^{i})^{\ast} = \sum_{i}p_{b} S_{ab}^{i} p_{a} x p_{a} (S_{ab}^{i})^{\ast} p_{b}.
\]
Since the span of the Kraus operators is independent of the Kraus decomposition, it follows that $S_{ab}p_{a} = p_{b} S_{ab}p_{a}$. Combining all the blocks, we conclude that $Sp = pSp$, i.e. $S$ admits a common invariant subspace, hence the algebra it generates cannot be equal to $B(H)$.
\end{proof}
Using Lemma \ref{Lem:irred}, in the undirected case we already see most of the statements appearing in Theorem \ref{Thm:mainequivalence}. We just need to cover the equivalences pertaining to different types of homomorphisms between quantum graphs. We will rephrase the definition from \cite{MRV} to resemble that of \cite{Matsuda2}. Then we will show that in the special case that is of interest to us these definitions coincide.
\begin{prop}
Let $\G_1 := (\M_1, \psi_1, A_1)$ and $\G_2 := (\M_2, \psi_2, A_2)$ be undirected quantum graphs. A unital $\ast$-homomorphism $f: \M_2 \to \M_1$ satisfies $f A_2 f^{\ast} \bullet A_1 =  A_1 $ if and only if it is a homomorphism from $\G_1$ to $\G_2$ in the sense of \cite[Definition V.4]{MRV}.
\end{prop}
\begin{proof}
We start by recalling the definition from \cite{MRV}. For a quantum adjacency matrix $A$ on $(\M,\psi)$ one can define a bimodular projection on $\M \otimes \M$ given by $P_{A}:= (m\otimes id)(id \otimes A \otimes id)(id \otimes m^{\ast})$. A $\ast$-homomorphism $f: \M_2 \to \M_1$ is a homomorphism from $\G_1$ to $\G_2$ if \begin{equation}
    P_{A_1} (f\otimes f) P_{A_2} = P_{A_1} (f\otimes f).\label{hom}
\end{equation}  Note that the maps on both sides are compositions of bimodular maps and $\ast$-homomorphisms, so if we denote these maps by $\Phi_1$ and $\Phi_2$, we get $\Phi_1(x\otimes y) = f(x) \Phi_1(\mathds{1} \otimes \mathds{1})f(y)$ and $\Phi_2(x\otimes y) = f(x) \Phi_2(\mathds{1} \otimes \mathds{1})f(y)$, hence it suffices to check equality at the unit $\mathds{1} \otimes \mathds{1}$.
%Note that both sides are compositions of bimodular maps and $\ast$-homomorphisms, so if we denote either of these maps by $\Phi$, we get $\Phi(x\otimes y) = f(x) \Phi(\mathds{1} \otimes \mathds{1})f(y)$, hence it suffices to check equality at the unit $\mathds{1} \otimes \mathds{1}$.

The right-hand side evaluated at the unit is equal to $(A_1\otimes id)m^{\ast}(\mathds{1})$, which is the Choi matrix of $A_1$ (see Definition \ref{def:Choi}). The left-hand side is a bit more complicated. The formula is 
\[(m\otimes id)(id \otimes A_1 \otimes id)(id \otimes m^{\ast})(f\otimes f)(A_2 \otimes id)m^{\ast}(\mathds{1}).
\]
We use the equality $(id \otimes f)m^{\ast}(\mathds{1}) = (f^{\ast} \otimes id)m^{\ast}(\mathds{1})$, which is just the adjoint of the equality $\psi\circ m(f\otimes id) = \psi \circ m(id \otimes f^{\ast})$; this equality says that the adjoint of $f$ is equal to the transpose of $f$, as $f$ is $\ast$-preserving. Then $(f\otimes f)(A_2\otimes id)m^{\ast}(\mathds{1})=(f\otimes id)(A_2\otimes id)(id\otimes f)m^{\ast}(\mathds{1})=(f\otimes id)(A_2\otimes id)(f^{\ast}\otimes id)m^{\ast}(\mathds{1})$. We can therefore rewrite our formula as follows
\[(m\otimes id)(id \otimes A_1 \otimes id)(id \otimes m^{\ast})(fA_2 f^{\ast} \otimes id)m^{\ast}(\mathds{1}).
\]
We now use the coassociativity of $m^{\ast}$, i.e. $(id \otimes m^{\ast}) m^{\ast} = (m^{\ast}\otimes id)m^{\ast}$ to arrive at
\[
(m\otimes id)(id \otimes A_1 \otimes id)(fA_2 f^{\ast} \otimes id\otimes id)(m^{\ast}\otimes id)m^{\ast}(\mathds{1}).
\]
This is equal to the Choi matrix of $fA_2 f^{\ast}\bullet A_1$,  so we conclude that $fA_2 f^{\ast} \bullet A_1 = A_1$ if and only if \eqref{hom} holds.
\end{proof}
\begin{prop}\label{prop:homomorphismconnected}
Let $\G := (\M, \psi, A)$ be a quantum graph and let $T:= (\C^2, \mu, id)$ be the trivial graph on two vertices. Then there exists a surjective graph homomorphism from $\G$ to $T$ in the sense of \cite[Definition V.4]{MRV} if and only it exists in the sense of \cite[Definition 3.1]{Matsuda2}. Both of these conditions are equivalent to reducibility of $A$.
\end{prop}
\begin{proof}
We know from the previous proposition that in both definitions there is an injective,
 unital $\ast$-homomorphism $f: \mathbb{C}^2 \to \M$, such that either $ff^{\ast} \bullet A = A$ (\cite{MRV}) or $id \bullet f^{\ast}Af = f^{\ast}Af$ (\cite{Matsuda2}). The injective, unital $\ast$-homomorphism is just a choice of two non-trivial projections $p_1, p_2 \in \M$ such that $p_1 + p_2 = \mathds{1}$, and then $f(e_i) = p_i$. It follows that $f^{\ast}(x) = \sum_{i=1}^{2} e_i \psi(p_i x)$, therefore $ff^{\ast}(x) = \sum_{i=1}^{2} p_i \psi(p_i x)$; this map is completely positive because $\psi$ is tracial. Note that the equality $ff^{\ast} \bullet A = A$ is equivalent to the equality  $\mathrm{Choi}(ff^{\ast}) \mathrm{Choi}(A) = \mathrm{Choi}(A)$ for Choi matrices (see Proposition \ref{prop:Choiproperties}). Let us compute $\mathrm{Choi}(ff^{\ast}) := (ff^{\ast}\otimes id)m^{\ast}(\mathds{1})$. Note that $ff^{\ast} = \sum_{i=1}^{2} L_{p_i} K L_{p_{i}}$, where $K(x) = \psi(x)\mathds{1}$ is the quantum adjacency matrix of the complete quantum graph. Therefore we obtain
 \[
 \mathrm{Choi}(ff^{\ast}) = \sum_{i=1}^{2} (L_{p_i} \otimes id)(K L_{p_i} \otimes id)m^{\ast}(\mathds{1}).
 \]
Since the Choi matrix of the transpose is equal to the tensor flip of the original Choi matrix, we get
\[
(K L_{p_i} \otimes id)m^{\ast}(\mathds{1}) = (id \otimes R_{p_i} K)m^{\ast}(\mathds{1}).
\]
Therefore $\mathrm{Choi}(ff^{\ast}) = \sum_{i=1}^{2}(L_{p_i} \otimes R_{p_i}) \mathrm{Choi}(K) = p_1 \otimes p_1^{op} + p_2 \otimes p_2^{op}$, because $\mathrm{Choi}(K)$ is the identity matrix in $\M \otimes \M^{op}$. The equality $ff^{\ast} \bullet A = A$ is thus equivalent to  $\mathrm{Choi}(A) = (p_1 \otimes p_1^{op} + p_2 \otimes p_2^{op}) \mathrm{Choi}(A)$.

 Assume that $A$ is not irreducible, so by Lemma \ref{Lem:irred} there are two projections $p_1$ and $p_2$ summing to $\mathds{1}$ such that $A(p_{i} x) 
 = p_i A(x)$. The Choi matrix $\mathrm{Choi}(A) = (A\otimes id)m^{\ast}(\mathds{1})$ satisfies $\mathrm{Choi}(A) = (A\otimes id)m^{\ast}(p_1^2 + p_2^2)$. Since $m^{\ast}$ is a bimodular map, this is equal to 
\[
\sum_{i=1}^{2}(A L_{p_i} \otimes R_{p_i})m^{\ast} (\mathds{1}) = \sum_{i=1}^{2}(L_{p_i} A \otimes R_{p_i})m^{\ast} (\mathds{1}),
\]
where we used the fact that $L_{p_i} A = A L_{p_{i}}$. It follows that $\mathrm{Choi}(A) = (p_1 \otimes p_1^{op} + p_2 \otimes p_2^{op}) \mathrm{Choi}(A)$, so $ff^{\ast} \bullet A = A$.

If $ff^{\ast} \bullet A = A$, then $\mathrm{Choi}(ff^{\ast}) \mathrm{Choi}(A) = \mathrm{Choi}(A)$, implies that $\mathrm{Choi}(ff^{\ast}) \geqslant \mathrm{Choi}(A)$, because $\mathrm{Choi}(A)$ is a projection. It follows that $ff^{\ast} \geqslant A$, where this is an inequality between completely positive maps. It follows that $p_i \psi(p_i) = ff^{\ast}(p_i) \geqslant A(p_i)$, so $A$ is not irreducible.

What happens in the definition of disconnectedness from \cite{Matsuda2}? It is equivalent for a quantum adjacency matrix $A$ to the equalities $\psi(p_1 A(p_2)) = \psi(p_2 A(p_1))=0$ (see \cite[Proof of Theorem 3.7]{Matsuda2}). Since we are in the tracial case, these can only vanish if $p_1 A(p_2) = p_2 A(p_1) = 0$, which by Lemma \ref{Lem:irred} is equivalent to $A(p x) = p A(x)$ (with $p_1=p$), therefore in this special case the definitions are equivalent.
\end{proof}
\begin{proof}[Proof of Theorem \ref{Thm:mainequivalence}]
Equivalence of \eqref{main1} and \eqref{main2} follows from Lemma \ref{Lem:irred}. Equivalence of \eqref{main3}
and \eqref{main4} follows from the combination of Lemma \ref{lem:Burnside} and \ref{lem:projectioninalgebra} (note that in the undirected case the algebra generated by $S$ is a $\mathrm{C}^{\ast}$-algebra). Proposition \ref{Prop:irredbimod} implies that \eqref{main1} and \eqref{main4} are equivalent, which completes the proof that the first four conditions are equivalent.

Condition \eqref{main4} is equvialent to \eqref{main5} because $\mathrm{C}^{\ast}(S)  = \text{span}\{S^k: k\in \mathbb{N}\}$ (which is closed in the finite dimensional case). 

By Proposition \ref{prop:homomorphismconnected} conditions \eqref{main6} and \eqref{main8} are equivalent to each other and to \eqref{main1}. They are also equivalent to \eqref{main7} by \cite[Theorem 4.9]{Matsuda2}.

Condition \eqref{main9} will be discussed in Proposition \ref{prop:Laplaciansimpleeigenvalue}, where we prove that it is equivalent to \eqref{main1} in a more general, non-tracial context.
\end{proof}

It turns out that there are plenty of examples of connected quantum graphs.
\begin{prop}
Let $\G$ be a random quantum graph from the model $QG(n,d)$ introduced in \cite{ChirWas} for $2\leqslant d \leqslant n^2-3$; $\G$ is modeled by a random $(d+1)$-dimensional operator subsystem of $M_n$. Then almost surely $\G$ is connected.
\end{prop}
\begin{proof}
It follows from the conditions on $d$ that the operator system $S$ of $\G$ contains two independent Hermitian matrices. It is well-known that such a pair generically generates the whole matrix algebra (see, e.g., \cite[Lemma 3.10]{ChirWas}), so in this case $\mathrm{C}^{\ast}(S) = M_n$, that is $\G$ is connected.
\end{proof}

\section{The non-tracial case}\label{Sec:nontracial}
In this section we assume that $\M$ is equipped with a (possibly) non-tracial functional $\psi$. Recall from Section \ref{sec:prelim} that we can define two inner products on $\M$, the usual GNS inner product and the KMS inner product. When $\tau$ is some faithful trace on $\M$  and $\psi(x) = \tau(\rho x)$, then the KMS inner product can be written as $\langle x,y\rangle_{\mathrm{KMS}} = \tau((\rho^{\frac{1}{4}}x\rho^{\frac{1}{4}})^{\ast} \rho^{\frac{1}{4}}y \rho^{\frac{1}{4}})$. It means that using the positive map $\iota(x):= \rho^{\frac{1}{4}} x \rho^{\frac{1}{4}}$ we can relate the KMS inner product to a tracial inner product. Because of that we will often work with \textbf{KMS implementations}, i.e. for a map $A:\M \to \M$ we will consider $\widetilde{A}(\rho^{\frac{1}{4}}x \rho^{\frac{1}{4}}):= \rho^{\frac{1}{4}}A(x) \rho^{\frac{1}{4}}$. It follows that $\widetilde{A}(x) = \rho^{\frac{1}{4}} A(\rho^{-\frac{1}{4}}x \rho^{-\frac{1}{4}}) \rho^{\frac{1}{4}}$. The advantage of this map is that it is equal to $\iota^{\ast} A \iota$, i.e. it is unitarily conjugate to $A$, but it acts on a tracial $L^2$-space.

In the non-tracial case we also need to adjust the left and right actions of $\M$ on itself. Let $L_x$ and $R_x$ be the operators of left and right multiplication by $x$. In the case of the GNS inner product we have $(L_x)^{\ast} = L_{x^{\ast}}$ but $(R_x)^{\ast} = R_{\sigma_{-i}(x^{\ast})}$. Therefore the right action needs to be twisted by the modular group in order for it to remain a $\ast$-homomorphism; the right regular representation becomes $\rho(x):= R_{\sigma_{-\frac{i}{2}}(x)}$. For the KMS inner product the left regular representation becomes $\lambda(x) = L_{\sigma_{\frac{i}{4}}(x)}$ and for the right regular representation $\rho(x)$ we need to add the modular group action $\sigma_{-\frac{i}{4}}(x)$. We will keep using the notation $L_x$ and $R_x$ for the left and right multiplication operators and $\lambda(x)$ and $\rho(x)$ for the left and right regular representations, both in the GNS and KMS cases. This leads us to the definition of connectedness in the non-tracial case. We phrase our results using the KMS inner product and then we will explain what happens for the GNS inner product.
\begin{defn}
Let $\mathcal{G}:=(\M, \psi, A)$ be an undirected (i.e. KMS symmetric) quantum graph. We say that it is \emph{disconnected} if there exists a non-trivial projection $p\in \M$ such that $\lambda(p)A = A \lambda(p)$, which means that $A$ commutes with the left regular representation of $p$. Equivalently, $\rho(p)A = A\rho(p)$, i.e. $A(x\sigma_{-\frac{i}{4}}(p)) = A(x) \sigma_{-\frac{i}{4}}(p)$ for all $x\in \M$.  
\end{defn}
\begin{rem}\label{Rem:KMSproj}
This definition reduces to Definition \ref{definition: connectivity 1} in the tracial case, because then $\lambda(p) = L_p$.

The condition from the definition is equivalent to $\widetilde{A}(px) = p\widetilde{A}(x)$ or $\widetilde{A}(xp) =\widetilde{A}(x) p$ for all $x\in \M$.  Indeed, let us check it for the second version: $\widetilde{A}(xp) =\widetilde{A}(x) p$ becomes
\[
\rho^{\frac{1}{4}} A(\rho^{-\frac{1}{4}}xp \rho^{-\frac{1}{4}}) \rho^{\frac{1}{4}} = \rho^{\frac{1}{4}} A(\rho^{-\frac{1}{4}} x \rho^{-\frac{1}{4}}) \rho^{\frac{1}{4}} p.
\]
We can multiply both from the left and right by $\rho^{-\frac{1}{4}}$ and denote $y = \rho^{\frac{1}{4}}x \rho^{\frac{1}{4}}$ to arrive at
\[
A(y \rho^{\frac{1}{4}}p \rho^{-\frac{1}{4}}) = A(y) \rho^{\frac{1}{4}}p \rho^{-\frac{1}{4}},
\]
which is precisely $A(y \sigma_{-\frac{i}{4}}(p)) = A(y) \sigma_{-\frac{i}{4}}(p)$. Since $\rho \in \M$, every $y\in \M$ is of the form $y=\rho^{\frac{1}{4}} x \rho^{\frac{1}{4}}$ for some $x\in \M$, so this establishes the equality for all $y\in \M$.
\end{rem}
The analogue of Lemma \ref{Lem:irred} for non-tracial functionals works because we adjusted the left and right regular representations to be $\ast$-homomorphisms.
\begin{prop}\label{prop:Laplaciansimpleeigenvalue}
Let $\mathcal{G}:=(\M,\psi, A)$ be an undirected quantum graph. Then it is connected if and only if $A$ is an irreducible map.
\end{prop}
\begin{proof}
It is easy to verify that $A$ is irreducible if and only if $\widetilde{A}$ is irreducible. If $\mathcal{G}$ is disconnected then by Remark \ref{Rem:KMSproj} $\widetilde{A}$ is reducible, so $A$ is reducible. Assume now that $A$ is reducible. By reducibility of $\widetilde{A}$ we find a projection $p$ such that $\widetilde{A}(xp) = \widetilde{A}(xp)p$, but $\widetilde{A}$ acts on a tracial $L^2$-space, hence by Lemma \ref{Lem:irred} we get $\widetilde{A}(xp) = \widetilde{A}(x)p$, which means that $\mathcal{G}$ is disconnected.
 \end{proof}
For the GNS inner product we can state a result not involving the modular group. Indeed, if a completely positive map is self-adjoint with respect to the GNS inner product, then it commutes with the modular group (see, e.g., \cite[Proposition 2.2]{Wirth}), in particular a GNS symmetric map is automatically KMS symmetric. It follows that the condition $A(x \sigma_{-\frac{i}{4}}(p)) = A(x) \sigma_{-\frac{i}{4}}(p)$ is equivalent to $A(y) p = A(yp)$, where $x= \sigma_{-\frac{i}{4}}(y)$. Therefore we can generalize \cite[Theorem 3.7]{Matsuda2} to the non-regular and non-tracial setting.
\begin{prop}\label{prop:Laplacianconnected}
Let $\mathcal{G}:= (\M, \psi, A)$ be a quantum graph with $A$ self-adjoint with respect to the GNS inner product. Then $\mathcal{G}$ is connected if and only if $0$ is a simple eigenvalue of the quantum graph Laplacian $\Delta$. 
\end{prop}
\begin{proof}
The quantum graph Laplacian can be constructed using the gradient, $\Delta= \nabla_{A}^{\ast} \nabla_{A}$ (see \cite[Remark 2.7]{Matsuda2}). It is clear that the kernel of the Laplacian is equal to the kernel of the gradient. By \cite[Proposition 2.3]{Matsuda2} we can identify the gradient with the commutator $[A,R_{(\cdot)}]$. It means that $y \in ker(\nabla_{A})$ iff $A(xy) = A(x)y$ for all $x\in \M$. %{\color{pink} Hence $\mathcal{G}$ is disconnected if and only if there is a non-trivial projection in $ker(\nabla_{A})$.} 
If $A$ commutes with $R_y$ then it also commutes with its adjoint $R_{\sigma_{-i}(y^{\ast})}$. From the fact that $A$ commutes with the modular group we conclude that it commutes with $R_{y^{\ast}}$, so the kernel is $\ast$-closed. It is also clearly an algebra (as the kernel of a derivation), so it is a $\mathrm{C}^{\ast}$-algebra. The unit belongs to the kernel, so $0$ is not a simple eigenvalue if and only if the kernel is a non-trivial $\mathrm{C}^{\ast}$-algebra. Since we are in a finite dimensional setting,  the $\mathrm{C}^{\ast}$-algebra $ker(\nabla_{A})$ is non-trivial if and only if it contains a non-trivial projection, which is equivalent to $\mathcal{G}$ being disconnected. 
\end{proof}
We will now show how to adjust the definition of graph homomorphisms (see \cite[Definition 3.1, Definition 3.2]{Matsuda2}) to the non-tracial case so that it still captures connectedness.
\begin{defn}\label{nontracial graph hom}
Let $\mathcal{G}_1:= (\M_1, \psi_1, A_1)$ and $\mathcal{G}_2:=(\M_2, \psi_2, A_2)$ be quantum graphs. A \emph{graph homomorphism} from $\mathcal{G}_1$ to $\mathcal{G}_2$ is a map $f: \M_2 \to \M_1$, whose \emph{KMS implementation} $\widetilde{f}$ is a $\ast$-homomorphism such that $A_2 \bullet (f^{\ast} A_1 f) = (f^{\ast} A_1 f)$, where the adjoint is computed with respect to the KMS inner product.

\end{defn}
Matsuda (\cite[Definition 3.2]{Matsuda2}) calls a quantum graph $\mathcal{G}:= (\M, \psi, A)$ \emph{disconnected} if it admits a surjective homomorphism onto the trivial graph $T_2$, meaning that the map $f: \mathbb{C}^2 \to \M$ is injective. We will show that it is equivalent to our definition.
\begin{prop}\label{prop:homomorphismdisconnected}
Let $\mathcal{G} = (\M, \psi, A)$ be an undirected quantum graph. Then it is disconnected if and only if it admits a surjective graph homomorphism onto $T_2$.
\end{prop}
\begin{proof}
In the tracial case the KMS implementation does not change anything. Let us see what happens in the non-tracial case. The KMS implementation of $f: \mathbb{C}^2 \to \M$ is equal to $\widetilde{f}(x):= \rho^{\frac{1}{4}}f(x) \rho^{\frac{1}{4}}$. An injective $\ast$-homomorphism from $\mathbb{C}^2$ is just a choice of two non-trivial projections $p_1$ and $p_2$ in $\M$ such that $p_1 + p_2 = \mathds{1}$. It means that $\widetilde{f}(e_i) = p_i$, i.e. $f(e_i) = \rho^{-\frac{1}{4}}p_i \rho^{-\frac{1}{4}}$. The KMS adjoint is equal to $f^{\ast}(x) = e_1 \tau(p_1 \rho^{\frac{1}{4}}x \rho^{\frac{1}{4}}) + e_2 \tau(p_2 \rho^{\frac{1}{4}}x \rho^{\frac{1}{4}})$. This computation yields $f^{\ast} A f(e_i) = e_1 \tau(p_1 \widetilde{A}(p_i)) + e_2 \tau(p_2 \widetilde{A}(p_i))$. After taking the Schur product with the identity (which is the adjacency matrix of the trivial graph $T_2$) we obtain $id \bullet f^{\ast} A f(e_i) = e_i \tau( p_i \widetilde{A}(p_i))$, so the graph homomorphism condition is that $\tau(p_1 \widetilde{A}(p_2)) = \tau(p_2 \widetilde{A}(p_1)) = 0$. But these are traces of products of positive operators, so $ \widetilde{A}(p_2) p_1 =  \widetilde{A}(p_1) p_2 = 0$. As $p_2 = \mathds{1} - p_1$, it follows from Lemma \ref{Lem:irred} that $\widetilde{A}$ is reducible, i.e. $\mathcal{G}$ is disconnected.
\end{proof}
\begin{rem}
It seems that even if we assume that $A$ is symmetric with respect to the GNS inner product, then we still have to adjust the definition of the graph homomorphism. The crucial point is that we concluded from the vanishing of the trace that the product of positive operators vanishes, which does not work for the GNS inner product. 
\end{rem}

We would also like to state the analogue of Proposition \ref{Prop:irredbimod} in the non-tracial case.
\begin{prop}\label{Prop:irredbimod nT}
Let $\mathcal{G}:= (\M,\psi, A)$ be a quantum graph and let $S \subset B(H)$ be the corresponding $\M'$-bimodule. Then $A$ is irreducible if and only if the algebra generated by $S$ is equal to $B(H)$.
\end{prop}
\begin{proof}
For $\M \simeq \bigoplus_{a=1}^{d} M_{n_{a}}$ we work with the representation on $H:= \bigoplus_{a=1}^{d} \mathbb{C}^{n_{a}}$. From Proposition \ref{prop:qadjacencyfrombimodule} we get a formula expressing the quantum adjacency matrix in terms of the bimodule $S$. It follows from it that the Kraus operators of the KMS implementation $\widetilde{A}$ can be taken to form an orthonormal basis of $S$. Therefore the proof of \ref{Prop:irredbimod} applies to show that the algebra generated by $S$ is equal to $B(H)$ if and only if $\widetilde{A}$ is irreducible. But $\widetilde{A}$ is irreducible if and only if $A$ is irreducible, so the proof is complete.
\end{proof}

\subsection{Bipartite quantum graphs}
Here we will generalize \cite[Theorem 3.8]{Matsuda2} to non-regular and non-tracial quantum graphs. 
\begin{defn}[{\cite[Definition 3.2]{Matsuda2}}]
Let $K_2$ be the irreflexive complete graph on two vertices, i.e. its adjacency matrix is equal to $\left[\begin{array}{cc} 0 & 1 \\ 1 & 0 \end{array}\right]$. We say that a quantum graph $\G$ is \emph{bipartite} if and only it admits a surjective graph homomorphism onto $K_2$. 
\end{defn}
\begin{rem} A very similar computation to the one in the proof of Proposition \ref{prop:homomorphismdisconnected} shows that a quantum graph is bipartite if and only if there exist two non-trivial projections $p_1$ and $p_2$ summing to $\mathds{1}$ such that $p_1 A(p_1) = 0 = p_2 A(p_2)$.
\end{rem}

Before we state and prove the results, we need an analogue of Lemma \ref{Lem:irred}; the proofs are completely analogous, so we provide fewer details here.
\begin{lemma}\label{Lem:bipartite}
Let $\Phi: \M \to \M$ be a completely positive map. Fix a non-trivial projection $p \in \M$. The following are equivalent:
\begin{enumerate}
\item\label{bipone} $\Phi(p) \leqslant C(\mathds{1} - p)$ and $\Phi(\mathds{1} - p)\leqslant Cp$ for some $C>0$;
\item\label{biptwo} $\Phi(p)p=0$ and $\Phi(\mathds{1} - p) (\mathds{1} - p)=0$;
\item\label{bipthree} $\Phi(xp) = \Phi(x)(\mathds{1}-p)$ for all $x\in \M$.
\end{enumerate}
\end{lemma}
\begin{proof}
It follows from \eqref{bipone} that $p\Phi(p)p =0$, hence $\Phi(p)p=0$ and $\Phi(\mathds{1} - p) (\mathds{1}-p) = 0$ can be proved in the same way. Let us now assume that \eqref{biptwo} holds. The Kadison-Schwarz inequality gives $(\Phi(xp))^{\ast} \Phi(xp) \leqslant C \Phi(p)$, hence $\Phi(xp)p=0$. We also get $\Phi(x(\mathds{1} - p))(\mathds{1} - p)=0$. It follows that 
\[
\Phi(xp) = \Phi(xp)(\mathds{1} - p) = \Phi(x)(\mathds{1}-p) - \Phi(x(\mathds{1}-p)) (\mathds{1}-p) = \Phi(x)(\mathds{1} - p).
\]
Now we assume \eqref{bipthree}. Equality $\Phi(xp) = \Phi(x)(\mathds{1}-p)$ for $x=\mathds{1}$ implies that $\Phi(p) \leqslant C(\mathds{1} - p)$ and for $x = \mathds{1}-p$ it shows that $\Phi(\mathds{1} - p)\leqslant Cp$.
\end{proof}
\begin{thm}\label{Thm:bipartite}
Let $\mathcal{G}:= (\M, \psi, A)$ be an undirected, connected quantum graph. Let $\lambda$ be the largest eigenvalue of $A$. The following are equivalent:
\begin{enumerate}
\item\label{bipar1} $-\lambda \in \sigma(A)$;
\item\label{bipar2} there is a non-trivial projection $p\in \M$ such that $\widetilde{A}(xp) =\widetilde{A}(x)(\mathds{1}-p)$ for all $x\in \M$, i.e. $\mathcal{G}$ is bipartite.
\end{enumerate}
\end{thm}
\begin{proof}
Suppose that \eqref{bipar2} holds. Let $u:= R_{2p-\mathds{1}}$; $u$ is a self-adjoint unitary. We will show that $u\widetilde{A} u = -\widetilde{A}$. By \eqref{bipar2} we have $\widetilde{A} R_{p} = R_{\mathds{1}-p} \widetilde{A}$ and by self-adjointness of $\widetilde{A}$ we have $\widetilde{A} R_{\mathds{1}-p} = R_{p} \widetilde{A}$ as well. It follows that $\widetilde{A} R_{2p-1} = - R_{2p-1} \widetilde{A}$, that is $u \widetilde{A} u = - \widetilde{A}$. Since $\widetilde{A}$ is similar to $-\widetilde{A}$, its spectrum has to be symmetric around zero, hence $-\lambda \in \sigma(\widetilde{A}) = \sigma(A)$. 

Assume now that $-\lambda \in \sigma(A) = \sigma(\widetilde{A})$. As $\widetilde{A}$ is $\ast$-preserving, we can find a %normalized
self-adjoint element $x\in \M$ such that $\widetilde{A}(x) = -\lambda x$ and $\|x\|_2=1$. Let $x= x_1 - x_2$ be the decomposition of $x$ into a difference of two positive elements such that $x_1 x_2=0$. We will show that $y:= x_1 + x_2$ is the unique eigenvector of $\widetilde{A}$ with eigenvalue $\lambda$ (unique by irreducibility/connectedness). We have 
\begin{align*}
\lambda &= | \tau \big(x \widetilde{A} (x)\big)| = |\tau\big( x_1 \widetilde{A}(x_1) + x_2 \widetilde{A}(x_2) - x_1 \widetilde{A}(x_2) - x_2 \widetilde{A}(x_1)\big)| \\
&\leqslant \tau\big( x_1 \widetilde{A}(x_1) + x_2 \widetilde{A}(x_2) + x_1 \widetilde{A}(x_2) + x_2 \widetilde{A}(x_1)\big) = \tau\big(y \widetilde{A}(y)\big) \leqslant \lambda,
\end{align*}
where the last inequality follows from the fact that $\|y\|_2 = \|x\|_2 = 1$, because $x_1 x_2 = 0$, and $\|\widetilde{A}\| = \lambda$. It means that we have an equality in the Cauchy-Schwarz inequality, hence $y$ is an eigenvector of $\widetilde{A}$ with eigenvalue $\lambda$, therefore $y$ has to be the Perron-Frobenius eigenvector, which is fully supported. It follows from equalities $\widetilde{A}(x_1 - x_2) = \lambda(x_2 - x_1)$ and $\widetilde{A}(x_1 + x_2) = \lambda(x_1 + x_2)$ that $\widetilde{A}(x_1) =\lambda x_2$ and $\widetilde{A}(x_2) = \lambda x_1$. We define $p$ to be the support of $x_1$, then $\mathds{1} - p$ is the support of $x_2$. It follows from  $\widetilde{A}(x_1) =\lambda x_2$ that $A(p) \leqslant C(\mathds{1} -p)$, so we conclude using Lemma \ref{Lem:bipartite}. 
\end{proof}
For GNS symmetric quantum adjacency matrices we get a result not involving the KMS implementations.
\begin{prop}
Let $A$ be a GNS symmetric, irreducible, completely positive map on $\M$ and let $\lambda$ be its largest eigenvalue. The following are equivalent:
\begin{enumerate}
\item $-\lambda \in \sigma(A)$;
\item there exists a nontrivial projection $p \in \M$ such that $A(xp) = A(x)(\mathds{1}-p)$ for all $x\in \M$.
\end{enumerate}
\end{prop}
\begin{proof}
A GNS symmetric map is KMS symmetric, so the previous result holds. We just need to translate the second condition into a statement about $A$. We have $\widetilde{A}(x(\mathds{1}-p)) = \widetilde{A}(x)p$. On the level of $A$ this equality translates to $A(x \sigma_{-\frac{i}{4}}(p)) = A(x) (\mathds{1} - \sigma_{-\frac{i}{4}}(p))$. Since $A$ commutes with the modular group, a computation similar to the one preceding Proposition \ref{prop:Laplacianconnected} shows that $A(xp) = A(x)(\mathds{1}-p)$.
\end{proof}
We also get an equivalent condition phrased in terms of the corresponding operator systems, analogous to Theorem \ref{Thm:mainequivalence}.
\begin{prop}
Let $\mathcal{G}$ be a bipartite quantum graph. Then there exists a non-trivial projection $p \in \M$ such that the corresponding operator system $S$ satisfies $pSp = (\mathds{1} - p)S(\mathds{1} - p) = \{0\}$.
\end{prop}
\begin{proof}
By Proposition \ref{prop:qadjacencyfrombimodule} the operator system $S$ is spanned by the Kraus operators of $\widetilde{A}$ (just like in the proof of Proposition \ref{Prop:irredbimod nT}). As $\mathcal{G}$ is bipartite, $\widetilde{A}$ satisfies $\widetilde{A}(xp) = \widetilde{A}(x) (\mathds{1}-p)$, so $\widetilde{A}(pxp) = (\mathds{1} - p)\widetilde{A}(x) (\mathds{1}-p)$. It follows that the span of the Kraus operators, i.e. $S$, satisfies $(\mathds{1} - p)S = Sp$, so $pSp = 0$ and $(\mathds{1}-p)S(\mathds{1}-p) = 0$.
\end{proof}

\subsection{Operator norm of regular quantum adjacency matrices}
A case left open in the work of Matsuda (see \cite[Theorem 2.11, Corollary 2.12]{Matsuda2}) is whether the operator norm of a $d$-regular quantum graph is equal to $d$. We prove here that indeed it is true.
\begin{prop}
Let $\G := (\M, \psi, A)$ be a $d$-regular quantum graph, i.e. $A\mathds{1} = d\mathds{1}$ and $A^{\ast}\mathds{1} = d\mathds{1}$, where $A^{\ast}$ is the GNS adjoint and $d\in [0,\infty)$. Then $\|A\|=d$, where $\|A\|$ is the operator norm on the GNS space.
\end{prop}
\begin{proof}
It suffices to show that $\|A\| \leqslant d$. Pick an element $x\in \M$ with $\psi(x^{\ast}x)\leqslant 1$ and we would like to show that $\|Ax\|^2 = \psi((A(x))^{\ast} A(x)) \leqslant d^2$.

Using the Kadison-Schwarz inequality, we get $(A(x))^{\ast} A(x) \leqslant \|A(\mathds{1})\| A(x^{\ast}x) = d A(x^{\ast}x)$. It follows that 
\begin{align*}
\psi((A(x))^{\ast} A(x)) &\leqslant d \psi(A(x^{\ast}x)) = d\langle \mathds{1}, A(x^{\ast}x)\rangle \\
&=d \langle A^{\ast}\mathds{1}, x^{\ast}x\rangle = d^2 \langle \mathds{1}, x^{\ast}x\rangle = d^2 \psi(x^{\ast}x) \\
&\leqslant d^2
\end{align*}
\end{proof}

\section{Acknowledgment}

This research was supported by the Deutsche Forschungsgemeinschaft (DFG, German Research Foundation) under Germany's Excellence Strategy – EXC 2044 – 390685587, Mathematics M\"{u}nster – Dynamics – Geometry – Structure, the Deutsche Forschungsgemeinschaft (DFG, German Research Foundation) – Project-ID 427320536 – SFB 1442, and ERC Advanced Grant 834267 – AMAREC, by the University of M\"{u}nster, Germany in the framework of the \textit{WiRe - Women in Research Fellowship Programme} and the National Science Center, Poland (NCN) grant no. 2021/43/D/ST1/01446. The project is co-financed by the Polish National Agency for Academic Exchange within the Polish Returns Programme. The second and third named author would also like to thank the Isaac Newton Institute for Mathematical Sciences, Cambridge, for support and hospitality during the programme \textit{Quantum Information, Quantum Groups and Operator Algebras}, where their collaborative work on this paper was initiated. This work was supported by EPSRC grant EP/Z000580/1. \includegraphics[scale=0.25]{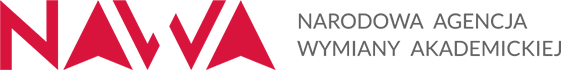}

%\nocite{*}
\bibliographystyle{amsalpha}
\bibliography{References_new}

%\bibliographystyle{plain} 
%\bibliography{References}

\end{document}